\documentclass[a4paper]{amsart}
\usepackage[utf8]{inputenc}
\usepackage[T1]{fontenc}
\usepackage{lmodern, enumerate}
\usepackage{amssymb,amsxtra}
\usepackage{amscd}
\usepackage[all]{xy}
\usepackage{xcolor}
\usepackage{nicefrac,mathtools}
\usepackage{microtype}
\usepackage{tikz-cd}
\usepackage{comment}
\usepackage[pdftitle={Semigroups},
 pdfauthor={Alcides Buss, Pere Ara, Ado Dalla Costa},
 pdfsubject={Mathematics}
]{hyperref}
\usepackage[lite]{amsrefs}

\numberwithin{equation}{section}
\theoremstyle{plain}
\newtheorem{theorem}[equation]{Theorem}
\newtheorem{lemma}[equation]{Lemma}
\newtheorem{proposition}[equation]{Proposition}

\theoremstyle{definition}
\newtheorem{definition}[equation]{Definition}
\theoremstyle{remark}
\newtheorem{remark}[equation]{Remark}
\newtheorem{example}[equation]{Example}

\newtheorem{notation}[equation]{Notation}

\newcommand*{\Z}{\mathbb Z}
\newcommand*{\N}{\mathbb N}

\newcommand{\Lab}{L^{{\rm ab}}}
\newcommand{\LVab}{LV^{{\rm ab}}}

\newcommand*{\ol}{\overline}

\newcommand{\mon}{\mathcal V}

\newcommand{\Ereg}{E^0_{\mathrm{reg}}}

\newcommand{\Tr}{\mathrm{Tr}}
\newcommand{\Idem}{\mathrm{Idem}}

\newcommand{\Ker}{\text{Ker}}

\begin{document}
\title[Weighted graphs]{Leavitt path algebras of weighted and separated graphs}

\author{Pere Ara}   \email{para@mat.uab.cat}
\address{Departament de Matem\`atiques, Edifici Cc, Universitat Aut\`onoma de Barcelona, 08193 Cerdanyola del Vall\`es (Barcelona), Spain, and}
\address{Centre de Recerca Matem\`atica, Edifici Cc, Campus de Bellaterra, 08193 Cerdanyola del Vall\`es (Barcelona), Spain.}

\begin{abstract}
In this paper we show that Leavitt path algebras of weighted graphs and Leavitt path algebras of separated graphs are intimately related.
We prove that any Leavitt path algebra $L(E,\omega)$ of a row-finite vertex weighted graph $(E,\omega)$ is $*$-isomorphic to the lower Leavitt path algebra of a certain bipartite separated graph $(E(\omega),C(\omega))$. For a general locally finite weighted graph $(E, \omega)$, we show that a certain quotient $L_1(E,\omega)$ of $L(E,\omega)$ 
is $*$-isomorphic to an upper Leavitt path algebra of another bipartite separated graph $(E(w)_1,C(w)^1)$. 
We furthermore introduce the algebra $\Lab (E,w)$, which is a universal tame $*$-algebra generated by a set of partial isometries.  
We draw some consequences of our results for the structure of ideals of $L(E,\omega)$, and we study in detail two different maximal ideals of the Leavitt algebra $L(m,n)$.  
\end{abstract}


\thanks{Partially supported by DGI-MINECO-FEDER grant PID2020-113047GB-I00, and the Spanish State Research Agency, through the Severo Ochoa and Mar\'ia de Maeztu Program for Centers and Units of Excellence in R$\&$D (CEX2020-001084-M)}

\maketitle



\section{Introduction}

A weighted graph is a pair $(E,\omega)$ consisting of a directed graph $E=(E^0,E^1,r,s)$ and a weight function $\omega \colon E^1\to \N$, where $\N$ is the set of positive integers. Leavitt path algebras of weighted graphs were introduced in \cite{Haz13} in order to obtain a graph theoretical model of Leavitt algebras $L (m,n)$ for arbitrary values $1\le m\le n$. Recall that Leavitt algebras were introduced by W. G. Leavitt in \cite{Lea62}, who showed that the (Leavitt) type of $L(m,n)$ is $(m,n-m)$. Some years later, Bergman found in \cite{Berg} the precise structure of the monoid $\mon (L(m,n))$ of isomorphism classes of finitely generated projective $L(m,n)$-modules. Recently, an interesting connection between the $K$-theory of Leavitt path algebras of weighted graphs and the theory of abelian sandpile models has been developed in \cite{AH}. We refer the reader to \cite{PrSurvey} for a nice survey on Leavitt path algebras of weighted graphs.

The algebras $L(m,n)$ were also described by the author and Ken Goodearl in \cite{AG2} as full corners of the Leavitt path algebras of the separated graphs $(E(m,n),C(m,n))$, and this was one of the key motivations to introduce this new type of algebras. Recall that a separated graph \cite{AG2} is a pair $(E,C)$ consisting of a directed graph $E$ and a partition $C$ of the set of edges of $E$ which refines the natural partition induced by the source function. 

	For all integers $1\leq m \leq n$, define the separated graph $(E(m,n),C(m,n))$ as follows: 
\begin{enumerate}
	\item $E(m,n)^0:=\{v,w\}$.
	\item $E(m,n)^1:=\{e_1,\ldots,e_n,f_1,\ldots,f_m\}$ ($n+m$ distinct edges).
	\item $s(e_i)=s(f_j)=v$ and $r(e_i)=r(f_j)=w$ for all $i,j$. 
	\item $C(m,n)=C(m,n)_v:=\{X,Y\}$, where $X=\{e_1,\ldots,e_n\}$ and $Y=\{f_1,\ldots,f_m\}$.
\end{enumerate}

By \cite[Proposition 2.12]{AG2}, we have an isomorphism between $L(m,n)$ and the corner algebra $wL(E(m,n),C(m,n))w$, hence $L(m,n)$ is isomorphic to a full corner of the Leavitt path algebra of the separated graph $(E(m,n),C(m,n))$. Since a full corner $eRe$ of a ring $R$ is Morita-equivalent to $R$, the rings $eRe$ and $R$ share many properties. For instance they have the same module theory and the same lattice of (two-sided) ideals.      

Observe that the graph $E(m,n)$ is a {\it bipartite graph}, that is, there is a partition of the set of vertices $E^0=E^{0,0}\sqcup E^{0,1}$ such that $s(E^1)\subseteq E^{0,0}$ and $r(E^1)\subseteq E^{0,1}$. We represent a bipartite separated graph $(E,C)$ by a diagram in which  we draw the vertices in $E^{0,0}$ in the upper level and the vertices in $E^{0,1}$ in the lower level of the diagram. According to this representation, we introduce in this paper the notions of the {\it upper Leavitt path algebra} $LV(E,C)$ and the {\it lower Leavitt path algebra} $LW(E,C)$ of a bipartite separated graph (see Section \ref{sec:weightedgraphs} for the precise definitions). Under the mild hypothesis that $s(E^1)=E^{0,0}$ and $r(E^1)=E^{0,1}$, if follows readily from the definitions that the upper and the lower Leavitt path algebras of a bipartite finitely separated graph $(E,C)$ are Morita-equivalent to the full Leavitt path algebra $L(E,C)$. It turns out that, in many examples, the significant algebra to consider is an upper or a lower Leavitt path algebra of a bipartite separated graph, see \cite[Section 9]{AE}.        

The purpose of this paper is to show that Leavitt path algebras of weighted graphs and Leavitt path algebras of separated graphs are intimately related.
A {\it vertex weighted graph} is a weighted graph $(E,\omega)$ such that $\omega (e)= \omega (f)$ for every pair of edges $e,f$ such that $s(e)=s(f)$. We will show in Section 2 that any Leavitt path algebra of a row-finite vertex weighted graph $(E,\omega)$ is $*$-isomorphic to the lower Leavitt path algebra of a certain bipartite separated graph $(E(\omega),C(\omega))$.
For a general row-finite weighted graph, we cannot construct a bipartite separated graph satisfying the property above, but we show in Section \ref{sect:algebraL1} that a certain quotient  $*$-algebra $L_1(E,w)$ of $L(E,w)$ is $*$-isomorphic to an {\it upper} Leavitt path algebra of another bipartite separated graph $(E(w)_1,C(w)^1)$. We furthermore introduce in Section \ref{sect:AbelianizedLPA} the algebra $\Lab (E,w)$, called the {\it abelianized Leavitt path algebra} of $(E,\omega)$, for any locally finite weighted graph $(E,w)$, and we show that it is $*$-isomorphic to a full corner of the $*$-algebra $\Lab (E(w)_1,C(w)^1)$ introduced in \cite{AE}. These abelianized algebras have a strong dynamical behaviour, being the crossed products of certain partial actions on totally disconnected Hausdorff topological spaces.  
Finally, we draw in Section \ref{sect:ideals} some consequences of our results for the structure of ideals of $L(E,\omega)$, shedding light on the second Open Problem in \cite[Section 12]{PrSurvey}. We illustrate our results by studying the ideals of the Leavitt algebras $L(m,n)$. In particular, two specific examples of maximal ideals of $L(m,n)$ are described.


\section{Bipartite separated graphs and weighted graphs}\label{sec:weightedgraphs}

We start with the definition of a separated graph. Concerning directed graphs, we will follow the conventions and notation in the book \cite{AAS}. In particular we will use the following definition of a path. Let $E=(E^0,E^1,r,s)$ be a directed graph. Then a trivial path (or path of length $0$) is just a vertex in $E$, and a non-trivial path is a sequence $e_1e_2\cdots e_n$ of edges in $E$ such that $r(e_i)= s(e_{i+1})$ for $i=1,\dots , n-1$. The extended (or double) graph of $E$, denoted by $\hat{E}$,  is the graph obtained from $E$ by adding a new edge $e^*$ for each edge $e\in E^1$, with $r(e^*)= s(e)$ and $s(e^*)= r(e)$.  

A {\it row-finite graph} is a graph $E$ such that $|s^{-1}(v)| <\infty$ for all $v\in E^0$. The set $\Ereg$ of {\it regular vertices} of a row-finite graph is the set of vertices $v$ such that $s^{-1}(v)\ne \emptyset$. A {\it locally finite graph} is a graph $E$ such that both $s^{-1}(v)$ and $r^{-1}(v)$ are finite, for all $v\in E^0$.

\begin{definition}(\cite[Definition 2.1]{AG2} and \cite[Definition 4.1]{AE})
	A \emph{separated graph} is a pair $(E,C)$ where $E$ is a (directed) graph and $C=\displaystyle\cup_{v \in E^0} C_v$ in which $C_v$ is a partition of $s^{-1}(v)$ into pairwise disjoint nonempty subsets for each vertex $v$. If all the sets in $C$ are finite, we say that $(E,C)$ is a \emph{finitely separated graph}. This is automatically true when $E$ is row-finite. 
	
	A \emph{bipartite separated graph} is a separated graph $(E,C)$ such that $E^0= E^{0,0}\sqcup E^{0,1}$, and $s(e)\in E^{0,0}$, $r(e)\in E^{0,1}$ for each $e\in E^1$.  
 \end{definition}

Note that our notation concerning ranges and sources of edges is the same as the one from \cite{Haz13}, \cite{PrSurvey}, \cite{AAS} and \cite{AG2}, but it is distinct from the one used in \cite{AE}, \cite{AL} and other sources.   

We can now define Leavitt path algebras of separated graphs, following \cite{AG2}. We will consider through the paper algebras and $*$-algebras over an arbitrary but fixed coefficient field $K$, endowed with an involution $*$.  
Note that our definitions usually refer to {\it presentations of $*$-algebras in the category of $*$-algebras}.

\begin{definition}\label{defC*grafosep}
	The \emph{Leavitt path algebra} of a separated graph $(E,C)$ with coefficients in $K$, is the *-algebra $L(E,C)$ with generators $\{v, e\mid v\in E^0, e\in E^1\}$, subject to the following relations: 
	\begin{enumerate}
		\item[(V)] $vv' = \delta_{v,v'}v$ and $v= v^*$ for all $v,v'\in E^0$,
		\item[(E)] $s(e)e = e= er(e)$ for all $e\in E^1$,
		\item[(SCK1)] $e^*f = \delta _{e,f}r(e)$ for all $e,f\in X$, $X\in C$, and 
		\item[(SCK2)] $v= \sum_{e\in X} ee^*$ for every finite set $X \in C_v$, $v\in E^0$.
	\end{enumerate}
\end{definition}

Note that the path algebra $P_K(\hat{E})$ of the extended graph $\hat{E}$ of $E$, endowed with its canonical involution, is precisely the $*$-algebra defined by relations (V) and (E), so that $L(E,C)$ is the quotient of $P_K(\hat{E})$ by the $*$-ideal corresponding to the relations (SCK1) and (SCK2).  

We will need the normal form of elements of $L(E,C)$, which was obtained in \cite{AG2}.

\begin{definition}
	For two non-trivial paths $\mu,\nu \in \text{Path}(E)$ with $s(\mu)=s(\nu)=v$ we say that $\mu$ and $\nu$ are \emph{$C$-separated} if the initial edges of $\mu$ and $\nu$ belong to different sets $X,Y \in C_v$. 
\end{definition} 

\begin{definition}
	For each finite $X \in C$, we select an edge $e_X \in X$. Let $\mu,\nu \in \text{Path}(E)$ be two paths such that $r(\mu)=r(\nu)$ and let $e$ and $f$ be the terminal edges of $\mu$ and $\nu$, respectively. The path $\mu \nu^*$ is said to be \emph{reduced} if $(e,f) \neq (e_X,e_X)$ for every finite $X \in C$. In case either $\mu$ or $\nu$ has length zero then $\mu\nu^*$ is automatically reduced.  
\end{definition}

\begin{theorem}\cite[Corollary 2.8]{AG2}\label{basegrafosseparados}
	Let $(E,C)$ be a separated graph. Then the set of elements of the form $$\mu_1\nu_1^*\mu_2\nu_2^* \ldots \mu_n\nu_n^*, \qquad  \mu_i,\nu_i \in \text{Path}(E)$$ such that  $\nu_i$ and $\mu_{i+1}$ are $C$-separated paths for all $i \in \{1,\ldots,n-1\}$ and $\mu_i\nu_i^*$ is reduced for all $i \in \{1,\ldots,n\}$ forms a linear basis of $L(E,C)$. We call $\mu_1\nu_1 \cdots \mu_n\nu_n^*$ a $C$-separated reduced path. 
\end{theorem}

For each edge $e$ of a separated graph $(E,C)$ we will denote by $X_e$ the unique element of $C$ such that $e\in X_e$. 

We are now ready for our definitions of the upper and lower path algebras of a bipartite separated graph. We first introduce these algebras using a presentation, and we show below in Proposition \ref{prop:LVandLW} that these are precisely the corner algebras of $L(E,C)$ corresponding to the upper and lower subsets $E^{0,0}$ and $E^{0,1}$ of $E^0$.

\begin{definition}\label{def:LV(E;C)}
	Let $(E,C)$ be a row-finite bipartite separated graph with $s(E^1) = E^{0,0}$, $r(E^1) = E^{0,1}$. Let $LV(E,C)$ be the universal $*$-algebra with generators $P_V \sqcup T$ where $P_V = \{ p_v\}_{v \in E^{0,0}}$ and $T = \{\tau (e,f)\}_{\{e,f \in E^1 \mid r(e) = r(f)\}}$, and subject to the relations
	\begin{enumerate}
		\item[(V')] $p_vp_{v'} = \delta_{v,v'}p_v$ and $p_v^*=p_v$ for all $v,v'\in E^{0,0}$,
	     \item[(T)] $\tau(e,f)^* = \tau(f,e)$,
		\item[(E')] $\tau(e,f) \cdot p_{s(f)} = p_{s(e)} \cdot \tau(e,f) = \tau(e,f)$,
		\item[(SCK1')] $\tau(e,f) \cdot \tau(g,h) = \delta_{f,g} \tau(e,h)$ for $f,g \in E^1$ belonging to the same set $Y \in C$,
		\item[(SCK2')] $p_v = \sum_{e \in Y} \tau(e,e)$ for all $v \in E^{0,0}$ and all $Y \in C_v$.
	\end{enumerate}
\end{definition} 

\begin{definition}\label{def:LW(E,C)}
	Let $(E,C)$ be a row-finite bipartite separated graph with $s(E^1) = E^{0,0}$, $r(E^1) = E^{0,1}$. Let $LW(E,C)$ be the universal $*$-algebra with generators $P_W \sqcup R$ where $P_W = \{ p_w\}_{w \in E^{0,1}}$ and $R = \{\rho(e,f)\}_{\{e,f \in E^1 \mid s(e) = s(f) \text{ and } X_e\ne X_f\}}$, and subject to the relations
	\begin{enumerate}
		\item[(V'')] $p_wp_{w'} = \delta_{w,w'}p_w$ and $p_w^*=p_w$ for all $w,w'\in E^{0,1}$,
		\item[(R)] $\rho(e,f)^* = \rho(f,e)$;
		\item[(E'')] $\rho(e,f) \cdot p_{r(f)} = p_{r(e)} \cdot \rho(e,f) = \rho(e,f)$,
		\item[(SCK1'')] 
		Suppose that $e,h\in E^1$ with $v:=s(e)= s(h)$ and $X\in C_v$ with $X\ne X_e$ and $X\ne X_h$. Then:
		$$\sum_{f\in X}\rho(e,f) \cdot \rho(f,h) = \begin{cases}
		\delta_{e,h}p_{r(e)} & \text{ if } X_e=X_h \\
		\rho(e,h) & \text{ if } X_e\ne X_h .
		\end{cases}$$
	\end{enumerate}
\end{definition} 

We will make use of the multiplier algebra $M(A)$ of an algebra $A$, see for instance \cite{APe} and \cite[Chapter 7]{ExelBook}, as a convenient way of defining our algebras. Note that $M(A)=A$ if $A$ is unital.  

Let $(E,C)$ be a row-finite bipartite separated graph such that $s(E^1)= E^{0,0}$ and $r(E^1)= E^{0,1}$. 
Let $V= \sum_{v\in E^{0,0}} v \in M(L(E,C))$ and $W=\sum_{w\in E^{0,1}} w\in M(L(E,C))$. Since the finite sums of vertices give a family of local units of $L(E,C)$, one can easily show that $V$ and $W$ exist in $M(L(E,C))$ and that $V+W=1$. It follows readily that $VL(E,C)V$ is linearly spanned by all the paths $\mu $ in $\hat{E}$ such that $s(\mu)\in E^{0,0}$ and $r(\mu)\in E^{0,0}$. A similar description holds for $LW(E,C)$. 

\begin{proposition}
	\label{prop:LVandLW}
	We have natural $*$-isomorphisms $\varphi_V\colon LV(E,C)\to VL(E,C)V$ and $\varphi_W\colon LW(E,C)\to WL(E,C)W$ such that 
	$$\varphi_V(p_v)= v, \quad \varphi_V(\tau (e,f))= ef^*,\quad \varphi_W(p_w)= w,\quad \varphi_W(\rho (e,f))= e^*f.$$
\end{proposition}

\begin{proof}
We will only show the isomorphism $LW(E,C)\cong WL(E,C)W$. The other isomorphism is proved in the same way.

One can easily see that the assignments $\varphi_W (p_w)= w$ and $\varphi_W (\rho(e,f)) = e^*f$ give a well-defined $*$-algebra homomorphism 
$$\varphi_W \colon LW(E,C)\to WL(E,C)W.$$
Clearly $\varphi_W$ is surjective. To show that $\varphi_W$ is injective we observe that using the defining relations of $LW(E,C)$, we can write each element of $LW(E,C)$ as a linear combination of terms of the form
$$\rho(e_1,f_1)\rho(e_2,f_2)\cdots \rho(e_n,f_n)$$
such that $s(e_i)=s(f_i)$, $X_{e_i}\ne X_{f_i}$ for $i=1,\dots ,n$, $r(f_i)= r(e_{i+1})$ and $f_ie_{i+1}^*$ is reduced for $i=1,\ldots ,n-1$.

Now suppose that $\alpha $ is a nonzero element of $LW(E,C)$ and that $\alpha = \sum \lambda_i \alpha_i$, where $\lambda_i\in K\setminus \{0\}$ and $\alpha_i$ are pairwise distinct terms as described in the above paragraph. Then $\varphi_W (\alpha )= \sum \lambda _i \varphi_W (\alpha_i)$ and $\varphi_W (\alpha_i)$ are pairwise distinct $C$-separated reduced paths in $L(E,C)$. Since these elements are linearly independent in $L(E,C)$, we conclude that $\varphi_W (\alpha )\ne 0$.  	
	\end{proof}

In view of Proposition \ref{prop:LVandLW}, we will identify the algebras $LV(E,C)$ and $VL(E,C)V$, and also the algebras $LW(E,C)$ and $WL(E,C)W$. 

We recall now the definitions of weighted graph and of Leavitt path algebra of a weighted graph, following \cite{PrSurvey}.

An {\it isolated vertex} of a graph $E$ is a vertex $v$ such that $s^{-1}(v)= r^{-1}(v)=\emptyset$. To avoid trivialities, we will restrict attention to graphs with no isolated vertices.

\begin{definition}
	\label{def:weighted-graph}
	A {\it weighted graph} is a pair $(E,\omega)$ where $E$ is a row-finite  graph with no isolated vertices,  and $\omega \colon E^1\to \N$ is a map. For each $v\in \Ereg$, set 
	$\omega (v):= \text{max}\{ \omega (e): e\in s^{-1}(v)\}$ and set $\omega (v)=0$ if $v$ is a sink. We say that $(E,\omega)$ is a {\it vertex weighted graph} if $w(e)=w(v)$ for all $e\in s^{-1}(v)$.  
\end{definition}

\begin{definition}
	\label{def:LPAweighted-graph}
	Let $(E,\omega)$ be a weighted graph and $K$ a field. The {\it weighted Leavitt path algebra} of $(E,\omega)$ is the $*$-algebra $L(E,\omega)$ generated  by $\{v,e_i,\mid v\in E^0, e\in E^1, 1\le i \le \omega (e)\}$ subject to the relations
	\begin{enumerate}
		\item $uv = \delta_{u,v}$ and $v= v^*$ for all $u,v\in E^0$,
		\item $s(e) e_i= e_i = e_ir(e)$, where $e\in E^1$, $1\le i\le \omega (e)$,
		\item $\sum_{e\in s^{-1}(v)} e_ie_j^* = \delta_{ij} v$, where $v\in E^0_{\text{reg}}$,  and $1\le i,j \le \omega (v)$,
		\item $\sum_{1\le i \le \omega (v)} e_i^*f_i = \delta _{e,f}r(e)$, where $v\in E^0_{\text{reg}}$ and $e,f\in s^{-1}(v)$.
	\end{enumerate}
 In relations (3) and (4) we set $e_i$ and $e_i^*$ zero whenever $i>\omega (e)$. 
\end{definition}

Now we consider a vertex weighted graph $(E,\omega)$ and we build a bipartite separated graph $(E(\omega),C(\omega ))$ such that $L(E,\omega)\cong LW(E(\omega),C(\omega))$.

\begin{definition}
	\label{def:separated-graph-of-weighted-graph}
	Let $(E,\omega)$ be a vertex weighted graph. Define a bipartite separated graph
	$(E(\omega),C(\omega))$ as follows. Let $V_0$ and $V_1$ be two copies of $E^0$, with bijections $E^0\to V_i$ given by $v\mapsto v_i$, for $i=0,1$. Set $E(\omega)^{0,0}=\{v_0\mid v\in \Ereg \}$, $E(\omega)^{0,1}=V_1$ and $E(\omega)^0= E(\omega )^{0,0}\sqcup E(\omega )^{0,1}$. For each
	$e\in E^1$, we define an edge $\tilde{e}\in E(\omega)^1$ such that $s(\tilde{e})= s(e)_0$ and $r(\tilde{e})= r(e)_1$. We set $X_v= \{\tilde{e}\mid e\in s^{-1}(v)\}$ for $v\in \Ereg $. In addition, we define another set of edges $Y_v= \{h(v,i): 1\le i \le \omega (v)\}$ for each $v\in \Ereg$, where  $s(h(v,i))= v_0$ and $r(h(v,i))= v_1$ for all $v\in \Ereg $ and all $1\le  i\le \omega (v)$. Finally, we define $C(\omega)_{v_0}= \{X_v,Y_v\}$ for each $v\in \Ereg$, and $E(\omega )^1= \bigsqcup_{v\in \Ereg} (X_v\sqcup Y_v)$. 
	
We call $(E(\omega ),C(\omega ))$ the {\it separated graph} of the vertex weighted graph $(E,\omega )$. Observe that, since $E$ has no isolated vertices, we have $s(E^1)= E^{0,0}$ and $r(E^1)= E^{0,1}$.  
	\end{definition}

\begin{theorem}
	\label{thm:isovertexweighted}
	Let $(E,\omega)$ be a vertex weighted row-finite graph. Then there is a unique isomorphism of $*$-algebras 
	$$\Phi = \Phi_{(E,\omega )}\colon L(E,\omega ) \longrightarrow LW(E(\omega ),C(\omega ))$$ such that $\Phi (v)= p_v= v_1$ for $v\in E^0$ and $\Phi (e_i) = \rho(h(s(e),i),\tilde{e})= h(s(e),i)^*\tilde{e}$, for $e\in E^1$ and $1\le i\le \omega (v)$.  
\end{theorem}

\begin{proof}
Write $(F,D)= (E(\omega ),C(\omega))$, $A=L(E,\omega)$ and $B=LW(E(\omega ),C(\omega))$. To show that $\Phi$ gives a well-defined $*$-algebra homomorphism, we need to check that the defining relations of $A=L(E,\omega)$ are satisfied in $B$. 
It is quite easy to show that relations (1) and (2) in Definition \ref{def:LPAweighted-graph} are preserved by $\Phi$. To show that (3) is also preserved, take $v\in E_{\text{reg}}^0$ and $1\le i,j\le \omega (v)$. We then have
\begin{align*}
	\sum _{e\in s^{-1}(v)} \Phi (e_i) \Phi (e_j)^* & = \sum _{e\in s^{-1}(v)} h(v,i)^*\tilde{e} \tilde{e}^* h(v,j) \\
	& = h(v,i)^*\Big( \sum _{e\in s^{-1}(v)} \tilde{e} \tilde{e}^*\Big) h(v,j) \\
	& = h(v,i)^*h(v,j) \\
	& = \delta_{i,j}v_1=\delta_{i,j} \Phi (v).
\end{align*}  
For (4), let $e,f\in s^{-1}(v)$, with $v\in E^0_{\text{reg}}$. Then we have
\begin{align*}
\sum _{1\le i\le \omega (v)} \Phi (e_i)^* \Phi (f_i) & = \sum _{1\le i\le \omega (v)} \tilde{e}^*h(v,i)h(v,i)^*\tilde{f}  \\
& = \tilde{e}^*\Big( \sum _{1\le i\le \omega (v)} h(v,i)h(v,i)^*\Big) \tilde{f}  \\
& = \tilde{e}^*\tilde{f} \\
& = \delta_{e,f}r(e)_1=\delta_{e,f} \Phi (r(e)).
\end{align*} 
Hence we have a well-defined $*$-homomorphism $\Phi\colon A\to B$. To build the inverse of $\Phi$, consider the map $\Psi \colon B\to A$ defined by 
$$\Psi (p_v) = v \quad (v\in E^0), \qquad \Psi (\rho(h(v,i),\tilde{e})) = e_i  \quad (e\in s^{-1}(v)).$$
	Here we observe that for each $v\in \Ereg$, we have $|D_{v_0}| = 2 $. Therefore, it follows from Definition \ref{def:LW(E,C)} that we only have  generators of the form $\rho(h(v,i),\tilde{e})$ and $\rho(\tilde{e}, h(v,i))=\rho(h(v,i),\tilde{e})^*$ for $e\in s^{-1}(v)$. Therefore to define $\Psi$ as a $*$-homomorphism, it is enough to define it on the given generators. 
	We need to show the preservation of the defining relations of $LW(F,D)$. We will only deal with (SCK1''). Let $v\in \Ereg$. There are two cases. Suppose first that $e=h(v,i)$ and  $h= h(v,j)$ for $1\le i,j\le \omega (v)$.
	Then the unique option for $X$ in (SCK1'') is $X= X_v$, and so
	\begin{align*}
\sum_{\tilde{f}\in X_v} & \Psi (\rho(h(v,i),\tilde{f})) \cdot  \Psi (\rho(\tilde{f},h(v,j))) = \sum_{\tilde{f}\in X_v} \Psi (\rho (h(v,i),\tilde{f})) \cdot  \Psi (\rho(h(v,j),\tilde{f}))^* \\
& = \sum_{f\in s^{-1}(v)} f_if_j^* = \delta_{i,j}v = \Psi (\delta_{i,j} p_{r(h(v,i))}). 
	\end{align*} 
	We now consider the second case. In this case, we have two edges $\tilde{e},\tilde{h}\in X_{v}$, that is, $e,f\in E^1$, with $s(e)=s(h)= v\in E^0$, and the unique possible value for $X$ in (SCK1'') is $X=Y_{v}$. We have
			\begin{align*}
	\sum_{1\le i \le \omega (v)}  & \Psi (\rho(\tilde{e}, h(v,i))) \cdot  \Psi (\rho(h(v,i),\tilde{h})) =  \sum_{1\le i\le \omega (v)} \Psi (\rho(h(v,i),\tilde{e}))^* \cdot  \Psi (\rho(h(v,i),\tilde{h}))  \\
	& = \sum_{1\le i \le \omega (v)} e_i^*h_i = \delta_{e,h}r(e) = \Psi (\delta_{e,h} p_{r(\tilde{e})}). 
	\end{align*} 
	 Hence $\Psi$ is a well-defined $*$-homomorphism. It is clear that $\Psi$ and $\Phi$ are mutually inverse. This concludes the proof.  
\end{proof}

\section{The algebra $L_1(E,\omega )$}
\label{sect:algebraL1}

In this section we will introduce a new $*$-algebra $L_1(E,\omega )$ associated to a row-finite weighted graph. This algebra is a certain quotient of $L(E,\omega)$, and has the property of being generated by partial isometries. This is not the case in general for the $*$-algebra $L(E,\omega )$.  We will show that if the graph $E$ is locally finite, then $L_1(E,\omega)$ is $*$-isomorphic to the {\it upper Leavitt path algebra} of a finitely separated graph. 

\begin{definition}
	\label{def:algebraL1}
	Let $(E,\omega)$ be a weighted graph and $K$ a field. The $*$-algebra $L_1(E,\omega)$ is the free $*$-algebra generated by $\{v,e_i\mid v\in E^0, e\in E^1, 1\le i \le \omega (e)\}$ subject to the relations
	\begin{enumerate}
		\item $uv = \delta_{u,v}$ and $v= v^*$ for all $u,v\in E^0$,
		\item $s(e) e_i= e_i = e_ir(e)$, where $e\in E^1$, $1\le i\le \omega(e)$,
		\item $e_ie_j^* =0$ for  all $e\in E^1$ and all $1\le i,j\le \omega (e)$ with $i\ne j$, 
		\item $e_i^*f_i = 0$ for all $e,f\in s^{-1}(v)$ with $e\ne f$, $v\in E^0$,  $1\le i\le \text{min}\{\omega (e),\omega (f)\}$.	
		\item $\sum_{e\in s^{-1}(v), \omega (e)\ge i} e_ie_i^* =  v$, where $v\in E^0_{\text{reg}}$,  and $1\le i \le \omega (v)$,
		\item $\sum_{1\le i \le \omega (e)} e_i^*e_i = r(e)$ for all $e\in E^1$. 
	\end{enumerate}
\end{definition}

\begin{remark}
	\label{rem:remarksondefinitionL1Ew}
	\begin{enumerate}
		\item 	Observe that relations (3) with $i\ne j$ and (4) with $e\ne f$ in Definition \ref{def:LPAweighted-graph} are automatically satisfied in $L_1(E,\omega )$ because of relations (3) and (4) above. Therefore we have
		$$L_1(E,\omega ) \cong L(E,\omega)/I_0,$$
		where $I_0$ is the $*$-ideal of $L(E,\omega )$ generated by the elements $e_i^*e_j$ for $e\in E^1$ and $1\le i\ne j\le \omega (e)$, and 
		$e_i^*f_i$ for $e,f\in s^{-1}(v)$ with $e\ne f$, $v\in E^0$,  $1\le i\le \text{min} \{ \omega (e), \omega (f)\}$.
		\item Notice that the elements $e_i$, for $e\in E^1$ and $1\le i\le \omega (e)$ are partial isometries in $L_1(E,\omega )$, that is, $e_ie_i^*e_i= e_i$. This follows by either right multiplying relation (5), or left multiplying relation (6), by $e_i$, and using relations (4) or (3) accordingly.  
	\end{enumerate}
\end{remark}

Let $(E,\omega)$ be a weighted graph and let $(E,\omega^M)$ be the unique vertex weighted graph such that $\omega^M(v)=\omega (v)$ for all $v\in E^0_{\text{reg}}$. Then $L(E,\omega)$ is the quotient $*$-algebra of $L(E,\omega^M)$ by the $*$-ideal generated by $e_j$, for $\omega (e)<j\le \omega (v)$, for each $v\in E^0$ and $e\in s^{-1}(v)$. The corresponding quotient $*$-algebra of $LW(E(\omega^M), C(\omega^M))$ is an object which cannot be exactly modeled with a separated graph.

However, through the consideration of a related graph, we will show that the $*$-algebra  $L_1(E,\omega) = L(E,\omega)/I_0$ is an {\it upper} Leavitt path algebra of a bipartite separated graph.

We need a definition from \cite{AE} (see also \cite{AL}).

\begin{definition}\label{definition-F.infty.and.others}
	Let $(E,C)$ be any locally finite bipartite separated graph, and write
	$$C_u = \{X_1^u,\ldots,X_{k_u}^u\}$$
	for all $u \in E^{0,0}$. Then the \textit{$1$-step resolution} of $(E,C)$ is the locally finite bipartite separated graph denoted by $(E_1,C^1)$, and defined by
	\begin{itemize}
		\item $E_1^{0,0} := E^{0,1}$ and $E_1^{0,1} := \{v(x_1,\ldots,x_{k_u}) \mid u \in E^{0,0}, x_j \in X_j^u\}$.
		\item $E^1 := \{\alpha^{x_i}(x_1,\ldots,\widehat{x_i},\ldots,x_{k_u}) \mid u \in E^{0,0}, i = 1,\ldots,k_u, x_j \in X_j^u\}$.
		\item $s(\alpha^{x_i}(x_1,\ldots,\widehat{x_i},\ldots,x_{k_u})) := r(x_i)$ and $r(\alpha^{x_i}(x_1,\ldots,\widehat{x_i},\ldots,x_{k_u})) := v(x_1,\ldots,x_{k_u})$.
		\item For $v\in E_1^{0,0}=E^{0,1}$, $C^1_v := \{X(x) \mid x \in r^{-1}(v)\}$, where
		$$X(x_i) := \{\alpha^{x_i}(x_1,\ldots,\widehat{x_i},\ldots,x_{k_u}) \mid x_j \in X_j^u \text{ for } j \ne i\}.$$
	\end{itemize}
	A sequence of locally finite bipartite separated graphs $\{(E_n,C^n)\}_{n \ge 0}$ with $(E_0,C^0) := (E,C)$ is then defined inductively by letting $(E_{n+1},C^{n+1})$ denote the $1$-step resolution of $(E_n,C^n)$. Finally set $(F_n,D^n)= \bigcup_{i=0}^n (E_i,C^i)$ and let $(F_{\infty},D^{\infty})$ be the infinite layer graph
	$$(F_{\infty},D^{\infty}) :=  \bigcup_{n=0}^{\infty} (F_n,D^n) =  \bigcup_{n=0}^{\infty} (E_n,C^n).$$
	It is clear by construction that $(F_{\infty},D^{\infty})$ is a separated Bratteli diagram in the sense of \cite[Definition 2.8]{AL}, called the {\it separated Bratteli diagram} of the locally finite bipartite graph $(E,C)$. (Note that only the case of a finite bipartite separated graph $(E,C)$ was considered in \cite{AE} and \cite{AL}. However the extension to locally finite bipartite separated graphs is straightforward.)
\end{definition}

By \cite[Theorem 5.1]{AE}, there is a canonical surjective $*$-homomorphism 
$$\phi _0\colon
L(E, C)\twoheadrightarrow L(E_{1}, C^{1}),$$
which is defined by 
$$\phi_0 (u) =  \sum _{(x_1,\dots ,x_{k_u})\in \prod_{i=1}^{k_u} X^u_i} v(x_1,\dots
,x_{k_u}), $$ where $C_u=\{ X_1^u,\dots , X^u_{k_u} \}$ for $u\in E^{0,0}$, $\phi
_0(w)=w$ for all $w\in E^{0,1}$, and 
$$\phi _0(x_i)= \sum _{x_j\in X^u_j, j\ne i} (\alpha
^{x_i}(x_1,\dots , \widehat{x_i},\dots ,x_{k_u}))^* \, $$
for an arrow $x_i\in X^u_i$. 

In order to state our next proposition, we need to extend the definition of the generators of $LW(E,C)$ for a bipartite separated graph $(E,C)$ to the case where $e,f$ belong to the same set of the partition $C$. 
Concretely, we set for $e,f\in X\in C$:
\begin{equation}
\label{eq:extendingrho}
\rho (e,f) = \delta_{e,f} r(e),\qquad (e,f\in X, X\in C).	
\end{equation}

\begin{proposition}
	\label{prop:kernel-of1step}
	The kernel of the $*$-homomorphism $\phi_0| \colon LW(E,C)\to LV(E_1,C^1)$ is the $*$-ideal $I$ of $LW(E,C)$ generated by the elements 
	$$\rho (e,f)\rho (f,g)\rho (g,h)- \rho (e,g)\rho(g,f)\rho (f,h)$$
	for all $e,f,g,h\in E^1$ such that $s(e)=s(f)= s(g)= s(h)$. Here $\rho(e,f)$ are the canonical generators of $LW(E,C)$ as given in Definition \ref{def:LW(E,C)} when $X_e\ne X_f$, and the elements given by \eqref{eq:extendingrho} when $X_e=X_f$. 
\end{proposition}

\begin{proof}
		Let $I$ be the $*$-ideal of $LW(E,C)$ generated by all the elements $\rho (e,f)\rho (f,g)\rho (g,h)- \rho (e,g)\rho(g,f)\rho (f,h)$, where $s(e)= s(f)= s(g)=s(h)$. 
	
	By \cite[Theorem 5.1]{AE}, the kernel of the map $\phi_0$ is the $*$-ideal $J$ of $L(E,C)$ generated by all the commutators $[ee^*,ff^*]$, for $e,f\in E^1$. Observe that 
	$$\rho (e,f)\rho (f,g)\rho (g,h)- \rho (e,g)\rho(g,f)\rho (f.h) = e^*[ff^*,gg^*]h \in J,$$ so we get 
	$I\subseteq WJW= \Ker (\phi_0|)$. For the other inclusion, observe that the family $e^*[ff^*,gg^*]h$, where $s(e)=s(f)=s(g)=s(h)$ is a family of generators for the ideal $WJW$. Indeed, we have $[ff^*,gg^*]=0$ except that $s(f)= s(g)$, so $J$ is generated as an ideal by the set of elements $[ff^*,gg^*]$, where $s(f)=s(g)$. Hence, every element of $WJW$ is a linear combination of terms of the form 
	$$\gamma [ff^*, gg^*]\lambda,$$
	where $v:= s(f)= s(g)$, $\gamma ,\lambda$ are paths in $\hat{E}$, with $r(\gamma)= v=s(\lambda)$ and $s(\gamma),r(\lambda)\in E^{0,1}$, so $\gamma = \gamma ' e^*$, $\lambda = h\lambda'$, with $e,h\in E^1$, $s(e)=s(h)=v$, and $\gamma',\lambda'\in LW(E,C)$, so that 
	$$\gamma [ff^*,gg^*]\lambda = \gamma ' (e^*[ff^*,gg^*]h)\lambda' \in  I.$$
	This concludes the proof. 
\end{proof}

We deduce from the above that $LW(E,C)/I\cong LV(E_1,C^1)$ through the $*$-homomorphism induced by $\phi_0$, where $I$ is the ideal generated by $\rho(e,f)\rho(f,g)\rho(g,h)- \rho (e,g)\rho (g,f)\rho (f,h)$, with $s(e)=s(f)= s(g)= s(h)$. With the help of Theorem \ref{thm:isovertexweighted} we can thus identify a suitable quotient of $L(E,\omega)$, for a vertex weighted graph $(E,\omega)$, which is an {\it upper} Leavitt path algebra of a separated graph. We show next that this algebra is precisely the algebra $L_1(E,\omega)$ from Definition \ref{def:algebraL1}.

Recall from Definition \ref{def:algebraL1} that $L_1(E,\omega)=L(E,\omega)/I_0$, where $I_0$ is the $*$-ideal of $L(E,\omega)$ generated by the elements:

\begin{equation}
\label{eq:sameedifferentij}
e_ie_j^*\,\, \quad  \text{for }\, e\in E^1 \,\text{ and }\,  1\le i,j\le \omega (e) \, \text{ with } \, i\ne j, 
\end{equation}

\begin{equation}
\label{eq:sameijdifferentef}	
e_i^*f_i  \,\, \quad  \text{for }\, e,f\in s^{-1}(v) \,\text{ with }\,   e\ne f, v\in E^0, \, \text{ and }\,  1\le i\le \text{min}\{\omega (e),\omega (f)\}.
\end{equation}

\begin{theorem}
	\label{thm:quotient-of-LEWM}
	Let $(E,\omega)$ be a locally finite vertex weighted graph, and let $L_1(E,\omega)=L(E,\omega)/I_0$ be the $*$-algebra from Definition \ref{def:algebraL1}. Then we have a canonical $*$-isomorphism
	$$\Phi _1 \colon L_1(E,\omega ) \longrightarrow LV(E(\omega )_1,C(\omega)^1)$$ 
	such that $\Phi_1 (v)= v_1$ and $\Phi_1 (e_i) = \tau(\alpha^{h(v,i)}(\tilde{e}), \alpha^{\tilde{e}}(h(v,i)))$ for $e\in E^1$, $1\le i \le \omega (e)$ and $v=s(e)$.  Hence $L_1(E,\omega)$ is $*$-isomorphic to the upper Leavitt path algebra of a separated graph.
\end{theorem}

\begin{proof}
	By Theorem \ref{thm:isovertexweighted}, we have a $*$-isomorphism
	$$\Phi \colon L(E,\omega )\to LW(E(\omega),C(\omega)).$$
	Composing with the surjective $*$-homomorphism $\phi_0| \colon LW(E(\omega ),C(\omega))\to LV(E(\omega)_1,C(\omega)^1)$ we obtain a surjective $*$-homomorphism
	$$\widetilde{\Phi}_1 \colon L(E,\omega) \to   LV(E(\omega)_1,C(\omega)^1).$$
	By Proposition \ref{prop:kernel-of1step}, we only have to check that $\Phi (I_0)= I$, where $I$ is the $*$-ideal of 
	$L(E(\omega),C(\omega))$ generated by all the elements $\gamma (a,b,c,d):=\rho (a,b)\rho(b,c)\rho (c,d)- \rho (a,c)\rho (c,b)\rho (b,d)$, for $a,b,c,d\in s^{-1}(v_0)$, $v\in E^0$. Given such an element, there are various possibilities to consider, depending on which elements $a,b,c,d$ belong to $X_v$ or to $Y_v$. Observe that $\gamma (a,b,c,d)= 0$ if $b$ and $c$ belong to the same set, so by symmetry we only need to consider the case where $b\in X_v$ and $c\in Y_v$. Assuming this, write $b=\tilde{f}$ and $c=h(v,i)$ for $f\in s^{-1}(v)$ and $1\le i\le \omega (v)$. We have four cases to consider:
	\begin{enumerate}
		\item $a= \tilde{e}\in X_v$ and $d=\tilde{g}\in X_v$. Then we have
		\begin{align*}
	\Phi^{-1}(\gamma (a,b,c,d) ) & = \Phi^{-1}( \tilde{e}^* \tilde{f} \tilde{f}^* h(v,i)h(v,i)^* \tilde{g} - \tilde{e}^* h(v,i)h(v,i)^* \tilde{f}\tilde{f}^* \tilde{g})\\
	& = \delta _{e,f} f_i^*g_i - \delta_{f,g} e_i^*f_i.
	\end{align*}
	This element in nonzero only if $e=f$ and $f\ne g$, or $e\ne f$ and $f=g$. In both of these cases we get an element of the form \eqref{eq:sameijdifferentef}. 
	\item Similarly, if $a,d\in Y_v$, then we get that $\Phi^{-1}(\gamma (a,b,c,d))$ is an element of the form \eqref{eq:sameedifferentij}. 
	\item $a=h(v,j)\in Y_v$ and $d=\tilde{e}\in X_v$. In this case, we have
	$$\Phi^{-1}(\gamma (a,b,c,d))= f_jf_i^*e_i - \delta_{i,j} \delta_{e,f} f_i.$$
	For $i=j$ and $e=f$ this gives the element $e_i e_i^*e_i-e_i$, which belongs to $I_0$ by Remark \ref{rem:remarksondefinitionL1Ew}(2). When $i\ne j$ or $e\ne f$ this gives an element which also belongs to $I_0$. 
	\item If $a\in X_v$ and $d\in Y_v$, then $\gamma (a,b,c,d)= - \gamma (d,b,c,a)^*$, and we reduce to case (3), 
\end{enumerate}

	 We thus obtain that $\Phi^{-1}(I)\subseteq I_0$, that is, $I\subseteq \Phi (I_0)$. The reverse inclusion $\Phi (I)\subseteq I_0$ follows easily from the above computations.  
	
	Hence we obtain a $*$-isomorphism $\Phi_1 \colon L_1(E,\omega)\to LV(E(\omega)_1,C(\omega)^1)$, with
	$\Phi _1 (v) = v_1$ and 
	\begin{align*}
	\Phi_1 (e_i) & = \phi _0 (h(v,i))^*\phi_0 (\tilde{e})\\
	& = \Big( \sum_{f\in s^{-1}(v)} \alpha^{h(v,i)}(\tilde{f})\Big)\Big( \sum_{j=1}^{\omega(v)} \alpha^{\tilde{e}}(h(v,j))^*\Big) \\
	&  =   \alpha^{h(v,i)}(\tilde{e})\alpha^{\tilde{e}}(h(v,i))^*,
	\end{align*}
	where the last equality follows from $r(\alpha^{h(v,i)}(\tilde{f})) = v(\tilde{f},i)$ and $r(\alpha^{\tilde{e}}(h(v,j))) = v(\tilde{e},j)$, which imply that the only nonzero summand is the one corresponding to $f=e$ and $j=i$. 
	
	This concludes the proof of the theorem. 
\end{proof}

It is time now to show that Theorem \ref{thm:quotient-of-LEWM} generalizes to arbitrary weighted graphs. The point is that in $L(E,\omega^M)/I_0$ we can kill the projections $e_je_j^*$ for $i+1\le j\le \omega(v)$, and this is equivalent to killing the projections $v(\tilde{e},h(v,j))$ in $L(E(\omega^M)_1,C(\omega^M)^1)$. 

Before proceeding to the statement of the next theorem, we will introduce a useful simplified notation for the generators of $L(E(\omega)_1,C(\omega)^1)$. This notation will be in use in the rest of the paper, whenever there is no danger of confusion. 

\begin{notation}
	\label{notation:LEw1Cw1}
	Let $(E,\omega)$ be a locally finite vertex weightewd graph, and let $R=L(E(\omega)_1,C(\omega)^1)$ be the Leavitt path algebra of the separated graph $(E(\omega)_1,C(\omega)^1)$. The elements of $E_1^{0,0}$ will be simply denoted by $v$, where $v\in E^0$. (These elements are denoted $v_1$ in the previous, general  notation.) The elements of $E^{0,1}_1$ will be denoted by $v(e,i)$, where $e\in E^1$ and $1\le i\le \omega (e)$. (These elements are denoted $v(\tilde{e}, h(v,i))$ in the previous, general notation.)
	The elements in $X(e)$ for $e\in E^1$ will be denoted by $\alpha^e(i)$, where $e\in E^1$ and $1\le i \le \omega (e)$. (These elements are denoted $\alpha^{\tilde{e}}(h(v,i))$ in the previous, general notation.)
	The elements in $X(v,i):= X(h(v,i))$ are denoted by $\alpha ^i(e)$, where $e\in E^1$ and $1\le i\le \omega (e)$. (These elements are denoted $\alpha^{h(v,i)}(\tilde{e})$ in the previous, general notation.)
	
	With the new notation, we have, by Theorem \ref{thm:quotient-of-LEWM}, 
	$$\Phi _1 (e_i) = \alpha^i(e)\alpha^{e}(i)^*= \tau (\alpha^i(e),\alpha^e(i))$$
	for $e\in E^1$ and $1\le i\le \omega (v)$.
\end{notation}

We can now extend the definition of the separated graph $(E(\omega)_1,C(\omega)^1)$ to any locally finite weighted graph.

\begin{definition}
	\label{def:sepgraphofarbitraryweightedgraph}
Let $(E,\omega )$ be a locally finite weighted graph. We define $(E(\omega )_1,C(\omega)^1)$ as the bipartite locally finite separated graph with 
$$E(\omega)_1^{0,0}= E^0,\qquad E(\omega )_1^{0,1}= \{v(e,i) : e\in E^1, 1\le i \le \omega (e) \},$$ and for each $v\in E^0$:
$$C(\omega )^1_v= \{X(v,1),\dots , X(v,\omega (v))\} \bigcup \{X(e)\mid e\in E^1, r(e)= v \},$$ where, for $1\le i\le \omega (v)$,
$$X(v,i)= \{\alpha^i (e) : e\in s^{-1}(v), \omega (e)\ge i \}, \quad s(\alpha^i(e)) = s(e)=v, \quad r(\alpha^i(e))= v(e,i),$$
and, for $e\in E^1$ with $r(e)=v$,
$$X(e) = \{ \alpha^e(i) : 1\le i\le \omega (e) \}, \quad s(\alpha^e(i)) = r(e)=v, \quad r(\alpha^e(i))= v(e,i).$$
\end{definition}

We now recall the definition of the quotient graph $(E/H,C/H)$, see \cite[Construction 6.8]{AG2}. We will use the pre-order $\le $ on $E^0$ given by $v\le w$ if and only if there is a path $\gamma $ in $E$ such that $s(\gamma)= w$ and $r(\gamma)= v$.  

\begin{definition}
	\label{def:hereditaryCsaturated}
	Let $(E,C)$ be a finitely separated graph. A subset $H$ of $E^0$ is {\it hereditary} if $v\le w$ and $w\in H$ imply that $v\in H$, and $H$ is $C$-saturated if whenever we have $v\in E^0$ such that $r(x)\in H$ for all $x\in X$, for some $X\in C_v$, then necessarily $v\in H$. Given a hereditary $C$-saturated subset $H$ of $E^0$, the quotient separated graph $(E/H,C/H)$ has $(E/H)^0 = E^0\setminus H$, and $(C/H)_v= \{X/C\mid X\in C_v\}$, where for $v\in E^0\setminus H$ and $X\in C_v$, $X/C:=\{ x\in X\mid r(x)\notin H \}\ne  \emptyset$. We denote by $\mathcal H(E,C)$ the lattice of hereditary $C$-saturated subsets of $E^0$. 
\end{definition}

\begin{theorem}
	\label{thm:IsoarbitraryWeightedGraph}
	Let $(E,\omega )$ be a locally finite weighted graph, let $L_1(E,\omega )$ be the $L_1$-algebra of $(E,\omega)$ (Definition \ref{def:algebraL1}), and let $(E(\omega )_1,C(\omega)^1)$ be the bipartite separated graph from Definition \ref{def:sepgraphofarbitraryweightedgraph}. Then there exists a canonical $*$-isomorphism 
	$$\Phi_1 \colon L_1(E,\omega) \overset{\cong}{\longrightarrow} LV(E(\omega)_1,C(\omega)^1)$$
	such that $\Phi_1 (v)= v$  for $v\in E^0$ and $\Phi_1 (e_i) = \tau(\alpha^{i}(e), \alpha^{e}(i))$ for $e\in E^1$ and $1\le i \le \omega (e)$.
	In particular $L_1(E,\omega )$ is isomorphic to the upper Leavitt path algebra of a separated graph. 
\end{theorem}

\begin{proof}
	By Theorem \ref{thm:quotient-of-LEWM}, we have a $*$-isomorphism
	$$\Phi_1\colon L_1(E,\omega^M)\longrightarrow LV(E(\omega^M)_1,C(\omega^M)^1)),$$
	where $(E,\omega^M)$ is the unique vertex weighted graph such that $\omega^M(v)=\omega (v)$ for all $v\in E^0$. 
	
	We consider the following subset of $E(\omega^M)_1^{0,1}$: 
	$$H = \{ v(e, j))\mid e\in E^1,  \omega (e)+1\le j \le \omega(s(e))\}.$$ 
	Then $H\subset E(\omega^M)_1^{0,1}$ is a hereditary subset of $E(\omega^M)_1^{0}$, because the vertices in $E(\omega^M)_1^{0,1}$ are sinks in $E(\omega^M)_1$. Let us check that $H$ is also $C(\omega^M)^1$-saturated. The sets in $C(\omega^M)^1$ are of one of the forms $X(v,i)$, for $1\le i\le \omega(v)$, or $X(e)$, for $e\in E^1$. It is enough to show that $r(X)\nsubseteq H$ for $X$ of these forms. We consider first the case $X=X(v,i)$, for $1\le i \le \omega (v)$. The elements of $X(v,i)$ are of the form $\alpha^{i}(e)$, where $e\in s^{-1}(v)$. Since $\omega (v)= \text{max}\{ \omega (e)\mid e\in s^{-1}(v)\}$, there is some $e\in s^{-1}(v)$ such that $\omega(e)=\omega (v)$, and then
	$$r(\alpha ^{i}(e))= v(e, i) \notin H.$$
	Now consider $e\in E^1$ and set $v= s(e)$. Since $\omega (e)\ge 1$, we have that 
	$r(\alpha^{e}(1))= v(e, 1)\notin H$.
	
	Hence we obtain that $H$ is a hereditary $C(\omega^M)^1$-saturated subset of $E(\omega^M)_1^0$.  The ideal $I(H)$ of $L(E(\omega^M)_1,C(\omega^M)^1)$ generated by $H$ satisfies that 
	$$L(E(\omega^M)_1,C(\omega^M)^1)/I(H)\cong L(E(\omega^M)_1/H,C(\omega^M)^1/H)$$
	(see the proof of \cite[Theorem 5.5]{AL}).

	We define $(E(\omega)_1,C(\omega)^1):= (E(\omega^M)_1/H,C(\omega^M)^1/H)$. 
	
	Observe that the $*$-isomorphism $L(E(\omega^M)_1,C(\omega^M)^1)/I(H) \to L(E(\omega)_1,C(\omega)^1)$  induces a $*$-isomorphism 
	$$\pi \colon LV(E(\omega^M)_1,C(\omega^M)^1)/IV(H) \longrightarrow LV(E(\omega)_1,C(\omega)^1),$$
	where $IV(H) = I(H)\cap LV(E(\omega^M)_1,C(\omega^M)^1)$. Let $\mathfrak I$ be the $*$-ideal of $LV(E(\omega^M)_1,C(\omega^M)_1)$ generated by the elements $\tau (\alpha^j(e),\alpha^e(j))$, for $\omega(e)+1\le j\le \omega(s(e))$. We claim that $\mathfrak I = IV(H)$. Clearly $\mathfrak I \subseteq IV(H)$. To show equality  observe first that $\mathfrak I$ contains the elements
	$$\tau (\alpha^j(e),\alpha^j(e))= \tau (\alpha^j(e),\alpha^e(j)) \tau (\alpha^j(e),\alpha ^e(j))^* $$
	and
	$$\tau (\alpha^e(j),\alpha^e(j)) = \tau (\alpha^j(e),\alpha^e(j))^* \tau (\alpha^j(e),\alpha ^e(j))$$ 
	for $\omega (v)+1\le j\le \omega (s(e))$. 
	We define a $*$-homomorphism $\varphi \colon LV(E(\omega)_1,C(\omega)^1)\to LV(E(\omega^M)_1,C(\omega^M)^1)/\mathfrak I$, by $\varphi (p_v) = \ol{p_v}$ for $v\in E^0$, and $\varphi (\tau(x,y))= \ol{\tau (x,y)}$, for $x,y\in (E(\omega)_1)^1$ with $r(x)= r(y)$. Here we indicate by $\ol{z}$ the class of an element $z$ in the quotient $LV(E(\omega^M)_1,C(\omega^M)^1)/\mathfrak I$. To see that $\varphi$ is well defined, we have to check that the relations in Definition \ref{def:LV(E;C)} are preserved by $\varphi$. This is obvious for all relations except for (SCK2'). Let us check relation (SCK2') for $Y= X(e)$, where $e\in E^1$. Since we are working in the graph $(E(\omega)_1,C(\omega)^1)$, this relation reads
	$$p_v= \sum_{1\le i\le \omega (e)} \tau (\alpha^e(i),\alpha^e(i)).$$
	Now, using that $\ol{\tau(\alpha^e(j),\alpha^e(j))} =0$ for $\omega(e)+1\le j \le \omega(s(e))$, we have
	$$\varphi \Big( \sum_{i=1}^{\omega(e)} \tau (\alpha^e(i),\alpha^e(i))\Big) = \sum_{i=1}^{\omega(s(e))} \ol{\tau (\alpha^e(i),\alpha^e(i)))} = \ol{p_v} = \varphi (p_v).$$
	Similarly, we can check that (SCK2')  is preserved for $Y=X(v,i)$ for $1\le i\le \omega(v)$. 
	Let 
	$$\zeta\colon LV(E(\omega^M)_1,C(\omega^M)^1)/\mathfrak I \longrightarrow LV(E(\omega^M)_1,C(\omega^M)^1)/IV(H)$$
	be the canonical quotient map. We clearly have the equality $\varphi \circ \pi \circ \zeta = \text{Id}_{LV(E(\omega^M)_1,C(\omega^M)^1)/\mathfrak I}$, hence $\zeta$ is injective, which implies that $\mathfrak I= IV(H)$, as desired. 
	
	Now the generators $\tau (\alpha^j(e),\alpha^e(j))$, for $\omega(e)+1\le j\le \omega(s(e))$, of the $*$-ideal $\mathfrak I = IV(H)$ correspond through the $*$-isomorphism $\Phi_1^{-1}$ to the elements $e_j$ in $L_1(E,\omega^M)$, and it is clear that
	$L_1(E,\omega)\cong L_1(E,\omega^M)/K$, where $K$ is the $*$-ideal of $L_1(\omega^M)$ generated by $e_j$, with $\omega(e)+1\le j\le \omega(s(e))$. We thus obtain a $*$-isomorphism, also denoted by $\Phi_1$, from $L_1(E,\omega)$ onto $LV(E(\omega)_1,C(\omega)^1)$, as desired.     
\end{proof}

\section{The algebra $\Lab (E,\omega)$}
\label{sect:AbelianizedLPA}

In this section we introduce the $*$-algebra $\Lab (E,\omega)$ for a locally finite weighted graph $(E,\omega)$, and we show it can be written as a full corner of a direct limit of a sequence of Leavitt path algebras of separated graphs. 

We begin with some preliminary definitions, see \cite[Definition 12.9]{ExelBook}.

\begin{definition}
	\label{def:tameset}
	A set $F$ of partial isometries of a $*$-algebra $A$ is said to be {\it tame} if for any two elements $u,u'\in U$, where $U$ is the multiplicative $*$-subsemigroup of $A$ generated by $F$, we have that $e(u)$ and $e(u')$ are commuting elements of $A$, where $e(v)= vv^*$ for $v\in U$. Note that if $F$ is a tame set of partial isometries, then all elements of $U$ are indeed partial isometries, and therefore the elements $e(u)$, with $u\in U$, are mutually commuting projections in $A$.   
	
	A $*$-algebra $A$ is said to be {\it tame} if it is generated as $*$-algebra  by a tame set of partial isometries. If $A$ is a $*$-algebra generated by a subset $F$ of partial isometries, there is a universal tame $*$-algebra $A^{\text{ab}}$ associated to $F$. By definition, there is a surjective $*$-homomorphism 
	$\pi \colon A \to A^{\text{ab}}$ so that $\pi (F)$ is a tame set of partial isometries generating $A^{\text{ab}}$, and such that for any $*$-homomorphism 
	$\psi \colon A \to B$ of $A$ to a $*$-algebra $B$ such that $\psi(F)$ is a tame set of partial isometries in $B$, there is a unique $*$-homomorphism
	$\ol{\psi}\colon A^{\text{ab}} \to B$ such that $\psi = \ol{\psi}\circ \pi$. 
	Indeed, we have 
	$A^{\text{ab}} = A/J$, where $J$ is the ideal of $A$ generated by all the commutators $[e(u),e(u')]$, where $u,u'\in U$.
	
	If $(E,C)$ is a separated graph, then the universal tame $*$-algebra associated to the generating set $E^0 \sqcup E^1$ of partial isometries of $L(E,C)$ is denoted by $\Lab (E,C)$. 
	
	Similarly, if $(E,C)$ is a row-finite bipartite separated graph, the  universal tame $*$-algebra associated to the generating set $E^{0,0} \sqcup \{\tau (x,y): x,y\in E^1, r(x)= r(y)\}$ of partial isometries of $LV(E,C)$ is denoted by $\LVab (E,C)$. 
\end{definition}
	
	We can now define the abelianized Leavitt path algebra of a weighted graph. Recall from Remark \ref{rem:remarksondefinitionL1Ew}(2) that the elements $e_i$, for $e\in E^1$ and $1\le i\le \omega(e)$ are partial isometries in $L_1(E,\omega)$. 
	
	\begin{definition}
		\label{def:abelLPAweighted}
		Let $(E,\omega)$ be a weighted graph and $K$ a field. The {\it abelianized Leavitt path algebra} of $(E,\omega)$, denoted by $\Lab (E,\omega)$ is the universal tame $*$-algebra associated to the 
		generating set $E^0\cup \{e_i: e\in E^1,1\le i\le \omega(e)\}$ of partial isometries of $L_1(E,\omega)$. Let $U$ be the $*$-subsemigroup of $L_1(E,\omega)$ generated by $\{e_i: e\in E^1,1\le i\le \omega(e)\}$, and let $J$ be the ideal of $L_1(E,\omega)$ generated by all the commutators $[e(u),e(u')]$, where $u,u'\in U$. Then $\Lab (E,\omega)= L_1(E,C)/J$. (Note that vertices commute with all elements $e(u)$.) 
	\end{definition}

Recall from Definition \ref{definition-F.infty.and.others} the canonical sequence $\{(E_n,C^n)\}_{n\ge 0}$ of locally finite bipartite separated graphs associated to a locally finite bipartite separated graph $(E,C)$. We denote by $\pi_n\colon L(E,C)\to L(E_n,C^n)$ and $\pi_{\infty}\colon L(E,C)\to \Lab (E,C)$ the canonical surjective $*$-homomorphisms. Note that $\pi_{2n}(V) = V_{2n}$, where $V_{2n}= \sum_{v\in E_{2n}^{0,0}} v\in M(L(E_{2n},C^{2n}))$.  

For a locally finite weighted graph $(E,\omega)$, we want to relate $\Lab (E,\omega)$ with the algebra $\Lab (E(\omega)_1,C(\omega)^1)$, where $(E(\omega)_1,C(\omega)^1)$ is the bipartite separated graph introduced in Section \ref{sect:algebraL1}. For this we need the following general lemma.

\begin{lemma}
	\label{lem:charaxcterizingLVab}
	Let $(E,C)$ be a locally finite bipartite separated graph, and set $V= \sum_{v\in E^{0,0}} v\in M(L(E,C))$. Let $\{(E_n, C^n)\}_{n\ge 0}$ be the canonical sequence of bipartite separated graphs associated to $(E,C)$. Then we have
$$\LVab (E,C)	\cong \pi_{\infty} (V)\Lab (E,C)\pi_{\infty} (V) \cong \varinjlim \pi_{2n} (V)L(E_{2n},C^{2n})\pi_{2n}(V).$$
\end{lemma}

\begin{proof}
	By \cite[Theorem 5.7]{AE}, we have
	$$\Lab (E,C) \cong \varinjlim L(E_{2n},C^{2n})$$
	naturally, and hence
	$$\pi_{\infty} (V)\Lab (E,C) \pi_{\infty}(V)  \cong \varinjlim \pi_{2n}(V)L(E_{2n},C^{2n})\pi_{2n}(V) .$$
	
	By definition, $\LVab (E,C)$ is the universal tame $*$-algebra with respect to the set 
	$$E^{0,0} \sqcup \{\tau (x,y): x,y\in E^1, r(x)= r(y)\}$$
	of partial isometries of $LV(E,C)=VL(E,C)V$ (see Proposition \ref{prop:LVandLW}). Let $U$ be the $*$-subsemigroup of $LV(E,C)$ generated by $F:= \{\tau (x,y): x,y\in E^1, r(x)= r(y)\}$, and let $J$ be the ideal of $LV(E,C)$ generated by all the commutators $[e(u),e(u')]$, with $u,u'\in U$. Let $\pi \colon LV(E,C)\to \LVab (E,C)=LV(E,C)/J$ be the projection map.
	
	 Since $\tau (x,y)= xy^*$ in $L(E,C)$, it is clear that $\pi _{\infty} (U)$ becomes a tame set of partial isometries in $\pi_{\infty} (V)\Lab (E,C) \pi_{\infty}(V)$. Hence there is a unique surjective $*$-homomorphism 
	$$ \mu \colon \LVab (E,C) \longrightarrow  \pi_{\infty} (V)\Lab (E,C) \pi_{\infty}(V)$$ 
	such that $\pi_{\infty}| = \mu \circ \pi$. It remains to show that $\mu$ is injective. 
	
	Let $U'$ be the $*$-subsemigroup of $L(E,C)$ generated by $E^1$, and let $J'$ be the ideal generated by all the commutators $[e(u),e(u')]$ for $u,u'\in U'$. 
	Then the kernel of $\pi_{\infty}$ is the ideal $J'$, and our task is to show that $J'\cap VL(E,C)V= J$. As we observed above, $J\subseteq J' \cap VL(E,C)V = VJ'V$. For the reverse inclusion, note that $VJ'V$ is generated by the following types of elements:
	\begin{enumerate}
		\item[(a)] Elements of the form $[e(u),e(u')]$, where 
		$$u=x_1x_2^* x_3x_4^*\cdots x_{2n-1}, \qquad u' = y_1y_2^*y_3y_4^*\cdots y_{2m-1}$$
		with $x_i,y_j\in E^1$ and $s(x_1)=s(y_1)$.
		\item[(b)] Elements of the form $z[e(u),e(u')]t^*$, where 
		 	$$u=x_1^*x_2 x_3^*x_4\cdots x_{2n-1}^*x_{2n}, \qquad u' = y_1^*y_2y_3^*y_4\cdots y_{2m-1}^*y_{2m}$$
		 with $z,t,x_i,y_j\in E^1$ and $r(x_1)=r(y_1)= r(z)= r(t)$.
	\end{enumerate} 
	For elements $u,u'$ as in (a), observe that $e(u)=e(ux_{2n-1}^*)$ and 
	$e(u')= e(u'y_{2m-1}^*)$, so that the corresponding element $[e(u),e(u')]$ belongs to $J$. 
	For elements $u,u'$ as in (b), note that $x_1[e(u),e(u')]x_1^*= [e(u_2),e(u_3)]$, where 
	$$u_2=x_1u=(x_1x_1^*)(x_2x_3^*)\cdots (x_{2n-2}x_{2n-1}^*)x_{2n},$$  
	$$u_3= x_1u'= (x_1y_1^*)(y_2y_3^*)\cdots (y_{2m-2}y_{2m-1}^*)y_{2m},$$
	so that $u_2,u_3$ are as in (a), so we get that $x_1[e(u),e(u')]x_1^*\in J$. But now
	$$z[e(u),e(u')]t^* = (zx_1^*)(x_1[e(u_2),e(u_3)]x_1^*)(x_1t^*) \in J.$$
We conclude that $VJ'V\subseteq J $, as desired. This completes the proof.	
	 \end{proof}

\begin{theorem}
	\label{thm:Labfor-weightedgraph}
	Let $(E,\omega)$ be a locally finite weighted graph, and let $(E(\omega)_1,C(\omega)^1)$ be the corresponding bipartite separated graph (Definition \ref{def:sepgraphofarbitraryweightedgraph}). Then there are natural $*$-isomorphisms
	\begin{equation}
	\label{eq:isos-for-Labs-maintheorem}
	\Lab (E,\omega) \cong \LVab (E(\omega)_1,C(\omega)^1) \cong \ol{V}\Lab (E(\omega)_1,C(\omega)^1)\ol{V},
	\end{equation}
	where $\ol{V}= \pi _{\infty}(V)\in M(\Lab (E(\omega)_1,C(\omega)^1))$.
	\end{theorem}

\begin{proof} Let $\Phi_1\colon L_1(E,\omega)\to LV(E(\omega)_1,C(\omega)^1)$ be the $*$-isomorphism from Theorem \ref{thm:IsoarbitraryWeightedGraph}. It is clear that $\Phi_1$ sends the canonical set of generators of $L_1(E,\omega)$ to the canonical set of generators of $LV(E(\omega)_1,C(\omega)^1)$. Hence we get that $\Phi_1$ induces a $*$-isomorphism $\Phi^{\text{ab}}\colon \Lab (E,\omega)\to \LVab (E(\omega)_1,C(\omega)^1)$. The second isomorphism in \eqref{eq:isos-for-Labs-maintheorem} follows from Lemma \ref{lem:charaxcterizingLVab}.
		\end{proof}
 
We finally point out that the representation of $\Lab (E,C)$ as a partial crossed product 
$$\Lab (E,C) \cong C_K(\Omega (E,C))\rtimes \mathbb F$$
for a finite bipartite separated graph (\cite[Corollary 6.12(1)]{AE}) can be easily adapted to obtain a corresponding representation for the abelianized Leavitt path algebra $\Lab (E,\omega)$ of any {\it finite} weighted graph.
Although we think that a suitable version holds for any locally finite weighted graph, we restrict here to the finite case, because only finite bipartite separated graphs are considered in \cite{AE}.

We refer the reader to \cite{ExelBook} and \cite{AE} for the background definitions on crossed products of partial actions of groups.

We first state the universal property of the dynamical system associated to a weighted graph. This is basically an internal version of the corresponding property for the associated bipartite separated graph $(E(\omega)_1, C(\omega)^1)$.  

\begin{definition}
	\label{thm:def:universalproperty}
	Let $(E,\omega)$ be a finite weighted graph. An $(E,\omega)$-dynamical system consists of a compact Hausdorff space $\Omega  $, with a family of clopen sets $\{\Omega _v\}_{v\in E^0}$ and, for each $v\in E^0$, a familiy of clopen subsets $H_{\alpha^i(e)}$, for $e\in s^{-1}(v)$ and $1\le i \le \omega (e)$, and  $H_{\alpha^e(i)}$, for each $e\in r^{-1}(v)$ and $1\le i \le \omega (e)$, such that 
	$$\Omega _v = \bigsqcup_{e\in s^{-1}(v):\omega (e)\ge i} H_{\alpha^i(e)},\qquad (v\in E^0, 1\le i\le \omega (v)),$$
	$$\Omega _v = \bigsqcup_{i=1}^{\omega (e)} H_{\alpha^e(i)}, \qquad (v\in E^0, e\in r^{-1}(v)),$$
	together with a family of homeomorphisms $\theta _{e_i}\colon H_{\alpha^e(i)} \to H_{\alpha^i(e)}$ for each $e\in E^1$ and $1\le i \le \omega (e)$.    
	
	Given two $(E,\omega)$-dynamical systems $(\Omega, \theta)$ and $(\Omega ',\theta')$, an {\it equivariant map} is a map $f\colon \Omega \to \Omega'$ such that $f(\Omega_v)\subseteq \Omega'_v$ for each $v\in E^0$ and $f(H_{\alpha^i(e)})\subseteq H'_{\alpha^i(e)}$, $f(H_{\alpha^e(i)})\subseteq H'_{\alpha^e(i)}$ for each $e\in E^1$ and $1\le i\le \omega (e)$, and $\theta'_{e_i}(f(x))= f(\theta_{e_i}(x))$ for all $x\in H_{\alpha^e(i)}$.  
	
	An $(E,\omega)$-dynamical system $(\Omega , \theta)$ is said to be {\it universal} if for each other $(E,\omega)$-dynamical system $(\Omega',\theta')$, there exists a unique equivariant continuous map
	$f\colon \Omega'\to \Omega$. Of course, if such a universal space exists, it is unique up to a unique equivariant homeomorphism.
	\end{definition}

\begin{theorem}
	\label{thm:crossed-product}
	Let $(E,\omega)$ be a finite weighted graph. Then there exists a universal $(E,\omega)$-dynamical system $(\Omega (E,\omega),\theta)$. Moreover $\theta $ can be extended to a partial action of the free group $\mathbb F$
	on $\{e_i: e\in E^1, 1\le i\le \omega (e)\}$, and we have
	$$\Lab (E,\omega) \cong C_K(\Omega(E,\omega))\rtimes _{\theta} \mathbb F.$$
	\end{theorem}

\begin{proof}
Let $(\Omega(E(\omega)_1,C(\omega)^1),\theta )$ be the universal $(E(\omega)_1,C(\omega)^1)$-dynamical system from \cite[Corollary 6.11]{AE}. (See \cite[Definition 6.10]{AE} for the definition of an $(E,C)$-universal dynamical system.)

Writing $\Omega:= \Omega(E(\omega)_1,C(\omega)^1)$, we have homeomorphisms 
$$\theta _{\alpha^i(e)}\colon \Omega_{v(e,i)}\to H_{\alpha^i(e)}, \quad \theta_{\alpha^e(i)}\colon \Omega_{v(e,i)} \to H_{\alpha^e(i)}$$
for each $e\in E^1$ and $1\le i \le \omega (e)$. 
It is easy to check that $\theta_{e_i}:= \theta_{\alpha^i(e)}\circ \theta_{\alpha^e(i)}^{-1}\colon H_{\alpha^e(i)}\to H_{\alpha^i(e)}$ provide the homeomorphisms making $\Omega (E,\omega): \bigsqcup_{v\in E^0} \Omega_v$ the universal $(E,\omega)$-dynamical system. 

The $*$-isomorphism $\Lab (E,\omega) \cong C_K(\Omega(E,\omega))\rtimes _{\theta} \mathbb F$ follows from the above observation, Theorem 4.4 and \cite[Theorem 6.12(1)]{AE}. 
	\end{proof}

\section{The $\mon$-monoid and structure of ideals}
\label{sect:ideals}

In this section, we study the $\mon$-monoid and the structure of ideals of the algebras $L_1(E,\omega)$ and $\Lab (E,\omega)$ introduced above. The results 
follow immediately from known results in \cite{AG2}, \cite{AE} and \cite{AL} and our previous work in the paper. We point out that the $\mon$-monoid of $L(E,\omega)$, where $(E,\omega)$ is an arbitrary weighted graph, has been determined by Preusser in \cite{PrIsr}. 

We refer the reader to \cite[Definition 3.2.1]{AAS} for the definition of the $\mon$-monoid $\mon (R)$ of a ring $R$. We will use the idempotent picture of $\mon (R)$, in which an element of $\mon (R)$ is given by the Murray-von Neumann equivalence class $[e]$ of an idempotent matrix $e$ over $R$.  

For a commutative monoid $M$, we will denote by $\mathcal L (M)$ its lattice of order-ideals, and for a (non-necessarily unital) ring $R$, we will denote by $\Tr (R)$ its lattice of trace ideals. See \cite[Section 10]{AG2} for these notions. It is shown in \cite[Theorem 10.10]{AG2} that for any ring $R$ there is a lattice isomorphism $\mathcal L (\mon (R))\cong \Tr (R)$. (The unital case of this result is due to Facchini and Halter-Koch \cite[Theorem 2.1(c)]{FH}.) We will denote by $\Idem (R)$ the lattice of idempotent-generated ideals of a ring $R$.

Let $R$ be a ring with local units (\cite[Definition 1.2.10]{AAS}), and let $e$ be an idempotent in the multiplier ring $M(R)$ of $R$. We say that $eRe$ is a (generalized) corner ring of $R$. The ring $eRe$ is a {\it full corner} of $R$ if $ReR=R$, that is, for each element $x\in R$ there are elements $r_i,s_i\in R$, $i=1,\dots , n$ such that $x=\sum_{i=1}^n r_ies_i$.

\begin{proposition}
	\label{prop:Moritaequivalence}
	Let $(E,C)$ be a row-finite bipartite separated graph such that $s(E^1) =
E^{0,0}$ and $r(E^1)= E^{0,1}$. Then both $LV(E,C)$ and $LW(E,C)$ are full corners of $L(E,C)$. In particular the inclusion $LV(E,C)\subseteq L(E,C)$ induces a monoid isomorphism 
	$$\mon (LV(E,C))\cong \mon (L(E,C)),$$ 
	and the usual restriction/extension process gives lattice isomorphisms $$\mathcal L (LV(E,C))\cong \mathcal L (L(E,C)),\qquad \Tr (LV(E,C)) \cong \Tr (L(E,C)).$$  Similar statements hold for  $LW(E,C)$ and $L(E,C)$,and also for the corresponding abelianized algebras $\LVab(E,C)$ and $\Lab (E,C)$.  	
	\end{proposition}

\begin{proof} Recall from Proposition \ref{prop:LVandLW} that $LV(E,C)=VL(E,C)V$ and $LW(E,C)=WL(E,C)W$, where $V=\sum_{v\in E^{0,0}} v\in M(L(E,C))$ and $W=\sum _{w\in E^{0,1}} w\in M(L(E,C))$. The fact that these are full corners follows immediately from the defining relations of $L(E,C)$, since $E^{0,0}=s(E^1)$ and $E^{0,1}= r(E^1)$.

Since all the involved algebras have local units, the result for the lattices of ideals follows from \cite[Proposition 3.5]{GS}. The proof for the $\mon$-monoids follows from the first paragraph of the proof of \cite[Lemma 7.3]{AF}.
	\end{proof}

We first define abstractly the monoid $M_1(E,\omega)$, and then we prove below
that
$M_1(E,\omega)$ is isomorphic to $\mon (L_1(E,\omega))$. 

\begin{definition}
	\label{def:Vmonoidweighted} 
	Let $(E,\omega)$ be a locally finite weighted graph. The monoid $M_1(E,\omega)$ is the commutative monoid with generators $\{ a_v: v\in E^0\}\cup \{a_{v(e,i)}: e\in E^1, 1\le i\le \omega (e)\}$ with the defining relations
	\begin{equation}
	\label{eq:Vmon1}
	a_v = \sum_{e\in s^{-1}(v), \omega (e)\ge i} a_{v(e,i)} \qquad (v\in \Ereg, 1\le i\le \omega (v)) 
	\end{equation}
	\begin{equation}
	\label{eq:Vmon2}
	a_{r(e)} = \sum _{i=1}^{\omega (e)} a_{v(e,i)} \qquad \,\, \,\,  (e\in E^1).
	\end{equation}
	\end{definition}

Observe that relations \eqref{eq:Vmon1} and \eqref{eq:Vmon2} give a refinement of the relation $\omega (v)a_v = \sum_{e\in s^{-1}(v)} a_{r(e)}$ for each $v\in \Ereg$. Observe also that there is a well-defined monoid homomorphism
$$\gamma_{(E,\omega)}\colon M_1(E,\omega) \longrightarrow \mon (L_1(E,\omega))$$
given by $\gamma (a_v)= [v]$ and $\gamma (a_{v(e,i)})= [e_ie_i^*]=[e_i^*e_i]$.

\begin{theorem}
	\label{thm:vmonoid}
	Let $(E,\omega)$ be a locally finite weighted graph. Then the natural homomorphism $\gamma_{(E,\omega)}\colon M_1(E,\omega)\to \mon (L_1(E,\omega))$ is an isomorphism.
\end{theorem}

\begin{proof}
Note that since $LV(E(\omega)_1,C(\omega)^1)$ is a full corner of $L(E(\omega)_1, C(\omega)^1)$, we have that the inclusion $\iota \colon LV(E(\omega)_1,C(\omega)^1)\subseteq L(E(\omega)_1, C(\omega)^1)$ induces an isomorphism of the corresponding $\mon$-monoids. It is clear that 
$M_1(E,\omega)=M(E(\omega)_1,C(\omega)^1)$ (see \cite[Definition 4.1]{AG2} for the definition of the graph monoid $M(E,C)$ of a separated graph). 

The composition of the maps
$$M_1(E,\omega)\overset{\gamma}{\to} \mon (L(E,\omega)) \overset{\mon(\Phi_1)}{\to} \mon (LV(E(\omega)_1, C(\omega)^1)) \overset{\mon(\iota)}{\to} \mon (L(E(\omega)_1, C(\omega)^1))$$
agrees with the canonical map $M(E(w)_1,C(w)^1)\to \mon (L(E(\omega)_1, C(\omega)^1))$, which is an isomorphism by 
\cite[Theorem 4.3]{AG2}. Since both $\mon (\Phi_1)$ and $\mon (\iota )$ are isomorphisms, we obtain that $\gamma $ is also an isomorphism.  
 	\end{proof}

\begin{theorem}
	\label{thm:lattice-isoLone}
	Let $(E,\omega)$ be a locally finite weighted graph. Then $\Idem (L_1(E,\omega))= \Tr (L_1(E,\omega))$, and we have lattice isomorphisms
	$$\Idem (L_1(E,\omega)) \cong \mathcal L (M_1(E,\omega)) \cong \mathcal H (E(w)_1,C(w)^1).$$
\end{theorem}

\begin{proof}
	Since $\gamma \colon M_1(E,C) \to \mon (L_1(E,\omega))$ is an isomorphism, 
	the monoid $\mon (L_1(E,\omega))$ is generated by equivalence classes of idempotents  in $L_1(E,\omega)$ and hence $\Tr (L_1(E,\omega))= \Idem (L_1(E,\omega))$ (see the proof of \cite[Proposition 6.2]{AG2}).
	By Theorem \ref{thm:vmonoid} and \cite[Theorem 10.10]{AG2}, we get
	$$\Idem (L_1(E,\omega)) =\Tr (L_1(E,\omega)) \cong \mathcal L(\mon (L_1(E,\omega))) \cong \mathcal L(M_1(E,\omega)).$$
Since $M_1(E,\omega)= M(E(\omega)_1, C(\omega)^1)$, it follows from \cite[Corollary 6.10]{AG2} that $\mathcal L (M_1(E,\omega)) \cong \mathcal H (E(\omega)_1,C(\omega)^1)$, completing the proof. 
	\end{proof}

A detailed study of the structure of ideals of the algebras $\Lab(E,C)$, for a finite bipartite separated graph $(E,C)$, has been performed in \cite{AL}. The main point is that the lattice of {\it trace ideals} of $\Lab(E,C)$ is isomorphic to the lattice of hereditary $D_{\infty}$-saturated subsets of $F_{\infty}^0$, where $(F_{\infty}, D^{\infty})$ is the separated Bratteli diagram of $(E,C)$ (Definition \ref{definition-F.infty.and.others}). Using these results and the fact that the lattices of ideals are preserved under Morita equivalence of rings with local units, one can translate all these results to the algebras $\Lab (E,\omega)$ for any finite weighted graph $(E,\omega)$. Observe that the ideals of $\Lab (E,\omega)$ can be pulled back to the original algebra $L(E,\omega)$, because $\Lab(E,\omega)$ is a quotient algebra of $L(E,\omega)$. 

 \begin{theorem}
	\label{thm:ideals-theorem}
	Let $(E,\omega)$ be a finite weighted graph. Let $(E(\omega)_1, C(\omega)^1)$ be the bipartite separated graph from Definition \ref{def:sepgraphofarbitraryweightedgraph}, and let $(E(\omega)_{\infty}, C(\omega)^{\infty})$ be the corresponding separated Bratteli diagram. Then there is are lattice isomorphisms 
	$$\Idem (\Lab (E,\omega))\cong \mathcal L (\mon (\Lab (E,\omega)))\cong \mathcal H (E(\omega)_{\infty}, C(\omega)^{\infty}).$$ 
\end{theorem}

\begin{proof}
	Apply Theorem \ref{thm:Labfor-weightedgraph} 
	and \cite[Theorem 4.5]{AL}.
	\end{proof}

\begin{remark}
	\label{rem:graded} 
	\begin{enumerate}
		\item[(a)] Observe that all ideals of $L(E,\omega)$ obtained by pulling back the ideals described in Theorem \ref{thm:ideals-theorem} are graded ideals. Hence we obtain a large family of graded ideals of $L(E,\omega)$, shedding some light on the second Open Problem in \cite[Section 12]{PrSurvey}.  
		\item Non-graded ideals of $\Lab (E,C)$ are studied in \cite[Section 7]{AL}. This gives information on non-graded ideals of $L(E,\omega)$. 
	\end{enumerate}
\end{remark}

\subsection{Remarks on the ideal structure of $L(m,n)$}
\label{subsect:ideals}
We close the paper with some remarks on the ideal structure of the Leavitt algebra $L(m,n)$. Since the algebra $L(1,n)$ is simple for all $n\ge 2$, we concentrate here on the remaining cases, so we will assume throughout this subsection that $1<m\le n$. Recall that $L(m,n)$ is the $*$-algebra with generators $x_{ij}$, $1\le i \le m$, $1\le j \le n$, subject to the relations given by the matricial equations $XX^*=I_m$, $X^*X=I_n$, where $X=(x_{ij})$ and $X^*$ is the $*$-transpose of $X$.    

It is an interesting and challenging problem to construct maximal ideals of $L(m,n)$, and study the corresponding simple factor rings. In particular we do not know any maximal ideal of $L(m,n)$ such that the corresponding simple factor algebra retains the Leavitt type $(m,n-m)$ of $L(m,n)$, although we suspect such maximal ideals exist. 

We will construct here two maximal ideals of $L(m,n)$, and we will relate one of them to the Leavitt path algebra of the minimal weighted graph of shape $(m,n)$, defined below.

Recall that a {\it partition} of a positive integer $l$ is a sequence of positive integers $\lambda = (\lambda_1,\lambda_2,\dots , \lambda _r)$ such
that $\lambda _i \ge \lambda_{i+1}$ for $i=1,\dots , r-1$ and $l=\lambda _1+\cdots +\lambda_r$. We will use the concept of the {\it shape} of a partition \cite[Definition 2.1.1]{Sagan}. 

\begin{definition}
	\label{def:weightedgraphshapemn} Let $1<m<n$ be integers. We say that $(E,w)$ is an $(m,n)$-{\it weighted graph} if it is a weighted graph with one vertex $v$, $n$ edges, and $w(v)= m$. 
\end{definition} 

The $(m,n)$-weighted graphs are completely determined, up to permutation of the edges, by its {\it shape}, which is constructed as follows. Choose an enumeration of the edges $e^{(1)},\dots , e^{(n)}$ of $E$ such that $m= w(e^{(1)})\ge w(e^{(2)})\ge \cdots \ge w(e^{(n)})\ge 1$. Then the {\it shape} of $(E,\omega)$ is the shape of the partition $(\lambda_1,\lambda_2,\dots ,\lambda_m)$ of $\lambda_1+\lambda_2+\cdots +\lambda_m$, where $\lambda_i:=|\{e\in E^1 : w(e)\ge i\}|$ for $i=1,\dots , m$. Observe that (up to permutation of edges) any partition $(\lambda_1,\lambda_2,\dots , \lambda_m)$ such that $\lambda_1=n$ determines a unique $(m,n)$-weighted graph, setting $w(e^{(i)})$ equal to the length of the $i$-th column of the shape of $\lambda$. We call such a partition an {\it $(m,n)$-partition}.   

Say that $\lambda = (\lambda_1,\lambda_2,\dots \lambda_m) \le \mu = (\mu_1,\mu_2,\dots ,\mu_m)$ if $\lambda_i\le \mu_i $ for all $i=1,2,\dots ,m$. (Observe that this is {\it not} the dominance ordering introduced in \cite[Definition 2.2.2]{Sagan}.) With this order, the set of $(m,n)$-partitions is a lattice, with maximum element $(n,n,\dots ,n)$ ($m$ times) and minimum elementy $(n,1,1,\dots ,1)$ (with $m-1$ one's).  

We may think of the shape of an $(m,n)$-partition $\lambda$ as a $\{0,1\}$ $m\times n$-matrix. For instance the shape of the $(3,4)$-partition $(4,2,2)$ 
is the matrix
$$A= \begin{pmatrix}
1 & 1 & 1 & 1 \\ 1 & 1 & 0 & 0 \\ 1 & 1 & 0 & 0 \end{pmatrix},$$
and we have $\omega (e^{(1)})=  \omega (e^{(2)})= 3$ and $\omega (e^{(3)})= \omega (e^{(4)})= 1$. 
We can also associate to $\lambda$ the corresponding matrix of the generators $(x_{ij})$ of $L(E,w)$, where $x_{ij}= e^{(j)}_i$, which is the matrix obtained from the full matrix 
$$\begin{pmatrix}
x_{11} & x_{12} & \cdots & x_{1n} \\
x_{21} & x_{22} & \cdots & x_{2n} \\
& \cdots & \cdots &     \\
x_{m1} & x_{m2} & \cdots & x_{mn}        	
\end{pmatrix}$$
by substituting by $0$ the variables which are not in the positions allowed by the shape of the partition. Looking at the algebra $L_1(E,\omega)$ of the weighted graph associated to the $(m,n)$-partition $\lambda$, we can interpret the shape of $\lambda$ as the refinement matrix $R$ which defines the $\mon$-monoid $M_1(E,\omega)$ of $L_1(E,\omega)$ (see Theorem \ref{thm:vmonoid}). For instance for the above partition $\lambda = (4,2,2)$, we have the refinement matrix
$$R= \begin{pmatrix}
	v(e^{(1)},1) & v(e^{(2)},1) & v(e^{(3)},1) & v(e^{(4)},1) \\ v(e^{(1)},2) & v(e^{(2)},2) & 0 & 0 \\ v(e^{(1)},3) & v(e^{(2)},3) & 0 & 0 \end{pmatrix}.$$
The sum of each row and each column of the matrix $R$ gives $a_v$ in the monoid $M_1(E,\omega)$ (by relations \eqref{eq:Vmon1} and \eqref{eq:Vmon2}), so that $R$ gives a refinement of the key identity $ma_v=na_v$ in $M_1(E,\omega)$. 

Let $\omega^M$ be the weight corresponding to the largest $(m,n)$-partition $\lambda = (n,n,\dots , n) =:(n^m)$. Of course $L(E,\omega^M) =L(m,n)$. 
By Theorem \ref{thm:lattice-isoLone}, the poset $\mathcal P$ of {\it proper} order-ideals of the monoid $\mon (L_1(E,\omega^M))$, which is isomorphic to the lattice of proper trace ideals of $L_1(E,\omega^M)$, is in bijective correspondence with the set of  $\{0,1\}$ $m\times n$ matrices having no zero rows and columns. 
The set of {\it maximal trace ideals} of $L_1(E,\omega^M)$ corresponds to the set of {\it minimal configurations}, which means that for each position $(i,j)$ with $a_{ij}= 1$, either row $i$ or column $j$ of the matrix $A$ contains only one $1$ (the one corresponding to the position $(i,j)$).    
The lattice of $(m,n)$-partitions embeds in an order-reversing way into the poset $\mathcal P$.   

We finish the paper by giving the construction of two different simple factor $*$-algebras of $L(m,n)$. The first already appears in \cite{AL}, and it is $*$-isomorphic to $L(1,n-m+1)$, so it is an ordinary Leavitt path algebra. The second is apparently new, and it is intimately related to the minimal $(m,n)$-partititon. This new factor algebra is not Morita equivalent to any (ordinary) Leavitt path algebra.

We start with the already known example.

\begin{example} (cf. \cite[Example 6.6]{AL})
	\label{exam:alreradyknown}
Let $\pi \colon L(m,n)\to L(1,n-m+1)$ be the surjective $*$-homomorphism given by the assignments  
\begin{align*}
\begin{pmatrix}
x_{11} & x_{12} & \cdots & x_{1n} \\
x_{21} & x_{22} & \cdots & x_{2n} \\
& \cdots & \cdots &     \\
x_{m1} & x_{m2} & \cdots & x_{mn}        	
\end{pmatrix} & \mapsto 
\begin{pmatrix}
\begin{matrix}
x_{11} & 0 & \cdots & 0\\
0 & x_{22} & \cdots & 0\\
& \cdots & \cdots &   \\
0 & 0 & \cdots & x_{m-1,m-1}
\end{matrix} & {\bf 0}_{(m-1)\times (n-m+1)}  \\
{\bf 0}_{1\times (m-1)} & x_{m,m} \,\,  x_{m,m+1} \,\, \cdots \,\,  x_{m,n}      	
\end{pmatrix}\\
& \mapsto
\begin{pmatrix}
I_{m-1} & {\bf 0}_{(m-1)\times (n-m+1)} \\
\,\, {\bf 0}_{1\times (m-1)}  & x_{1} \,\,  x_{2} \,\, \cdots \,\,  x_{n-m+1} 
\end{pmatrix}
\end{align*}
where $x_1,x_2,\dots , x_{n-m+1}$ are the standard generators of $L(1,n-m+1)$. Obviously, the homomorphism $\pi$ factors through $\Lab (m,n)$.
\end{example}

We consider now the second example.

\begin{example}
	\label{exam:maximal.idealminpart} Let $3 \le  m \le n$. Then 
	the algebra $L(m,n)$ has a maximal ideal $\mathfrak m$ such that 
	$$L(m,n)/\mathfrak m \cong L(1,m-1)\otimes L(1,n-1).$$
	In particular the quotient  $L(m,n)/\mathfrak m$ is not Morita equivalent to any Leavitt path algebra.
\end{example}

\begin{proof} 
	Let $\omega_0$ be the weight corresponding to the minimal $(m,n)$-partition $(n,1^{m-1})$. It corresponds to the following generating matrix 
$$X=\begin{pmatrix}
x_{11} & x_{12} & \cdots & x_{1n} \\
x_{21} & 0 & \cdots & 0 \\
x_{31} & 0  & \cdots & 0 \\
& \cdots & \cdots &     \\
x_{m1} & 0 & \cdots & 0        	
\end{pmatrix}$$
with relations $XX^*=I_m$ and $X^*X=I_n$. Observe that in this case we have $L(E,\omega_0)=L_1(E,\omega_0)$, so that we exactly recover the weighted Leavitt path algebra with our construction. 
The monoid $M_1(E,\omega_0)\cong \mon (L(E,\omega_0))$ is given by
$$M_1(E,\omega _0)= \langle a, x \mid a= x+(m-1)a= x+(n-1)a\rangle .$$
(Here $a$ correspond to $a_v$ and $x$ to $a_{v(e^{(1)},1)}$.) 
It is easy to see that the Leavitt type of $L(E,\omega_0)$ is $(1,n-m)$. Indeed observe that 
$$(n-m+1)a= a+(n-m)a = x+(m-1)a+(n-m)a= x+(n-1)a = a,$$
and the Grothendieck group of $M_1(E,\omega_0)$ is $\Z/(n-m)$ (with generator $a$), so that the Leavitt type is $(1,n-m)$ as claimed.

The algebra $L(E,\omega_0)$ has a unique non-trivial trace ideal $M$, corresponding to setting $x_{11}= 0$. We have $L(E,\omega_0)/M \cong L(1,m-1)\ast L(1,n-1)$, the coproduct of the two simple Leavitt algebras $L(1,m-1)$ and $L(1,n-1)$. This is indeed the Leavitt path algebra of the separated graph $(E',C')$ with a unique vertex $v$ and with $C'=\{X,Y\}$,
$X=\{x_1,\dots , x_{n-1}\}$ and $Y=\{y_1,\dots , y_{m-1}\}$. 
The $*$-isomorphism is given by the assignment
$$ \begin{pmatrix}
0 & x_{12} & \cdots & x_{1n} \\
x_{21} & 0 & \cdots & 0 \\
x_{31} & 0  & \cdots & 0 \\
& \cdots & \cdots &     \\
x_{m1} & 0 & \cdots & 0        	
\end{pmatrix} 
\mapsto  \begin{pmatrix}
0 & x_{1} & \cdots & x_{n-1} \\
y_{1}^* & 0 & \cdots & 0 \\
y_{2}^* & 0  & \cdots & 0 \\
& \cdots & \cdots &     \\
y_{m-1}^* & 0 & \cdots & 0        	
\end{pmatrix} 
$$
We then have  
$$\mon (L(E,\omega_0)/M) \cong \langle a \mid a= (n-1)a= (m-1)a\rangle $$ 
and we get another drop in the Leavitt type, because this algebra has Leavitt type $(1,d)$ where $d:= \text{gcd} (m-2,n-2)$. 

We obtain a maximal ideal $\mathfrak m '\supset M$  such that 
$L(E,\omega_0)/\mathfrak m '\cong L(1,m-1)\otimes L(1,n-1)$, and pulling back this ideal to $L(m,n)$ we obtain the desired ideal $\mathfrak m$. The last statement follows from \cite[Theorem 5.1]{AC}.  
\end{proof}




\end{document}